\documentclass{article}

\usepackage[all,ps,arc]{xy}
\usepackage{amsmath,amssymb,latexsym}
\usepackage{amsfonts}
\usepackage{amsthm}

\newtheorem{Theorem}{Theorem}[section]
\newtheorem{Lemma}[Theorem]{Lemma}
\newtheorem{Proposition}[Theorem]{Proposition}

\newtheorem{Remark}[Theorem]{Remark}
\newtheorem{Example}[Theorem]{Example}

\newcommand{\bR}{\ensuremath{\mathbb R}}
\newcommand{\bS}{\ensuremath{\mathbb S}}
\newcommand{\bX}{\ensuremath{\mathbb{X}}}
\newcommand{\bY}{\ensuremath{\mathbb{Y}}}
\newcommand{\bW}{\ensuremath{\mathbb{W}}}
\newcommand{\bZ}{\ensuremath{\mathbb{Z}}}
\newcommand{\bE}{\ensuremath{\mathbb{E}}}

\newcommand{\cC}{\ensuremath{\mathcal{C}}}
\newcommand{\cE}{\ensuremath{\mathcal{E}}}
\newcommand{\cF}{\ensuremath{\mathcal{F}}}
\newcommand{\cM}{\ensuremath{\mathcal{M}}}
\newcommand{\cX}{\ensuremath{\mathcal{X}}}

\newcommand\sfEq{\ensuremath{\mathsf{Eq}}}
\newcommand\EffEq{\ensuremath{\mathsf{EffEq}}}
\newcommand\Gpd{\ensuremath{\mathsf{Gpd}}}

\newcommand\Set{\ensuremath{\mathsf{Set}}}
\newcommand\Simpl{\ensuremath{\mathsf{Simpl}}}
\newcommand\Ab{\ensuremath{\mathsf{Ab}}}

\DeclareMathOperator{\Supp}{Supp}
\DeclareMathOperator{\Dec}{Dec}
\DeclareMathOperator{\Eq}{Eq}

\begin{document}

\title{A relative monotone-light factorization system for internal groupoids}
\author{Alan S. Cigoli, Tomas Everaert and Marino Gran}

\maketitle

\begin{center}
\textit{Dedicated to Robert Lowen on the occasion of his seventieth birthday}
\end{center}

\begin{abstract}
Given an exact category $\mathcal C$, it is well known that the connected component reflector $ \pi_0 \colon \Gpd(\cC) \rightarrow \cC$ from the category $\Gpd(\cC)$ of internal groupoids in $\cC$ to the base category $\cC$ is semi-left-exact. In this article we investigate the existence of a monotone-light factorization system associated with this reflector. We show that, in general, there is no monotone-light factorization system $(\cE',\cM^*)$ in \Gpd(\cC), where $\cM^*$ is the class of coverings in the sense of the corresponding Galois theory. However, when restricting to the case where $\cC$ is an exact Mal'tsev category, we show that the so-called comprehensive factorization of regular epimorphisms in \Gpd(\cC) is the relative monotone-light factorization system (in the sense of Chikhladze) in the category \Gpd(\cC) corresponding to the connected component reflector, where $\cE'$ is the class of final functors and $ \cM^*$ the class of regular epimorphic discrete fibrations.
\end{abstract}

\section*{Introduction} \label{sec:factsyst}

Every full reflective subcategory \cX\ of a finitely complete category \cC
\begin{equation} \label{diag:adj}
\xymatrix{
\cC \ar@<1.2ex>[r]^I_\bot & \cX \ar@<1.2ex>[l]^H
}
\end{equation}
gives rise to a prefactorization system $(\cE,\cM)$ where \cE\ is the class of arrows inverted by the reflector $I$. This type of prefactorization systems are extensively studied in \cite{CHK}, where conditions are given for $(\cE,\cM)$ to be a factorization system. This is always the case when $I$ is \emph{semi-left-exact}, i.e.\ when it preserves the pullbacks of arrows in \cM\ along any arrow. In this case, the factorization $f=me$ of an arrow $f$ is displayed by the diagram
\[
\xymatrix{
A \ar[dr]^e \ar@/^2ex/[drr]^{\eta_A} \ar@/_2ex/[ddr]_f \\
& B\times_{HI(B)}HI(A) \ar[r] \ar[d]^m & HI(A) \ar[d]^{HI(f)} \\
& B \ar[r]_-{\eta_B} & HI(B)
}
\]
where $m$ and the unlabelled arrow are the pullback projections, $\eta_A$ and $\eta_B$ are components of the unit of the adjunction, and the arrow $e$ is induced by the universal property of the pullback. So, in this case, the class \cM\ is given by those morphisms $m$ in \cC\ such that the diagram
\begin{equation} \label{diag:triv.cov}
\begin{aligned}
\xymatrix{
A \ar[r]^-{\eta_A} \ar[d]_m & HI(A) \ar[d]^{HI(m)} \\
B \ar[r]_-{\eta_B} & HI(B)
}
\end{aligned}
\end{equation}
is a pullback.

As explained in \cite{CJKP}, a reflection $I$ is semi-left-exact precisely when the adjunction (\ref{diag:adj}) is \emph{admissible} in the sense of Categorical Galois Theory \cite{J89,J90}. In this case, the morphisms in the class \cM\ above are exactly the \emph{trivial coverings} of the corresponding Galois Theory, whereas the class $\cM^*$ of \emph{coverings} is obtained via a localization process. More precisely, a morphism $f$ is in $\cM^*$ if there exists some effective descent morphism $p$ such that the pullback of $f$ along $p$ lies in \cM. As shown in \cite{CJKP}, there are interesting situations in topology and algebra where the class $\cM^*$ is part of a factorization system $(\cE',\cM^*)$, where $\cE'$ is the class of those arrows which are stably in \cE, i.e.\ whose pullbacks are always in \cE. This process of simultaneously ``stabilizing'' and ``localizing'' the pair $(\cE,\cM)$, first considered in \cite{CJKP}, does not give a factorization system, in general. When this is the case, one says that $(\cE',\cM^*)$ is the \emph{monotone-light} factorization system associated with $(\cE,\cM)$. This kind of factorization systems and their relationship with torsion theories in the context of normal categories are widely studied in \cite{EG13}.
\medskip

Every exact category \cC\ can be seen as a reflective subcategory of the category $\Gpd(\cC)$ of its (internal) groupoids,
\begin{equation} \label{adj:Gpd-C}
\xymatrix{
\Gpd(\cC) \ar@<1.2ex>[r]^-{\pi_0}_-\bot & \cC \ar@<1.2ex>[l]^-D
}
\end{equation}
where $D$ sends any object in \cC\ to the corresponding discrete groupoid, and $\pi_0$ is the connected component functor, which is a semi-left-exact reflector (see \cite{B87}). It turns out that the class $\cM^*$ for the corresponding Galois structure is given by (internal) discrete fibrations (see Proposition \ref{prop:df}), and since the latter are part of the well known \emph{comprehensive} factorization system \cite{SW, B87}, it is natural to ask if this is actually the monotone-light factorization system associated with the adjunction above.

Our counterexample \ref{final.not.pbs} shows that the answer to this question is negative, in general. However, as we explain in Section \ref{sec:ex.mal}, when \cC\ is an exact Mal'tsev category, one can see the comprehensive factorization  for regular epimorphisms in $\Gpd(\cC)$ as a \emph{relative} monotone-light factorization system (in the sense of \cite{Ch04}). This result relies in particular on the fact that any internal groupoid \bX\ admits a (relative) stabilising object \cite{CJKP} that is the image $\Dec(\bX)$ of \bX\ under the \emph{d\'ecalage} functor recalled in Section \ref{sec:conn.comp}.

\section{Galois structures and relative monotone-light factorization}

A Galois structure $(\cC,\cX,I,H,\eta,\epsilon,\cF,\Phi)$ (in the sense of \cite{J89,J90}) is a system satisfying the following properties:
\begin{enumerate}
\item $I \dashv H$ is an adjunction \eqref{diag:adj} with unit $\eta\colon 1_\cC\to HI$ and counit $\epsilon\colon IH \to 1_\cX$;
\item $\cF$ and $\Phi$ are subclasses of arrows in $\cC$ and in $\cX$, respectively, such that 
  \begin{enumerate}
  \item[(i)] $I(\cF)\subseteq\Phi$ and $H(\Phi)\subseteq\cF$;
  \item[(ii)] \cC\ admits pullbacks along arrows in \cF, and \cF\ is pullback stable, \\ \cX\ admits pullbacks along arrows in $\Phi$, and $\Phi$ is pullback stable;
  \item[(iii)] \cF\ and $\Phi$ contain all isomorphisms and are closed under composition.
  \end{enumerate}
\end{enumerate}
For each object $B$ in \cC, we denote by $\cF(B)$ the full subcategory of the ``slice category'' $\cC\downarrow B$ whose objects are in the class \cF, and similarly $\Phi(I(B))$ will denote the full subcategory of $\cX\downarrow I(B)$ whose objects are in $\Phi$. Then there is an induced adjunction
\[
\xymatrix{
\cF(B) \ar@<1.2ex>[r]^-{I^B}_-\bot & \Phi(I(B)) \ar@<1.2ex>[l]^-{H^B}
},
\qquad \eta^B\colon 1_{\cF(B)}\to H^BI^B,
\quad \epsilon^B\colon I^BH^B\to 1_{\Phi(I(B))},
\]
where $I^B$ is defined by the image under $I$, and $H^B$ by the pullback along $\eta_B$ of the image under $H$.

The object $B$ is said to be admissible if $\epsilon^B$ is an isomorphism. We shall say that the Galois structure $(\cC,\cX,I,H,\eta,\epsilon,\cF,\Phi)$ is \emph{admissible} if each $B$ in ${\cC}$ is admissible.

A morphism $p\colon E\to B$ in \cC\ is said to be a \emph{monadic extension} if the pullback functor $p^*\colon \cF(B)\to\cF(E)$ is monadic. An object $f\colon A\to B$ in $\cF(B)$ is said to be a \emph{trivial covering} when the diagram
\[
\xymatrix{
A \ar[d]_{f} \ar[r]^-{\eta_A} & HI(A) \ar[d]^{HI(f)} \\
B \ar[r]_-{\eta_B} & HI(B)
}
\]
is a pullback. An object $f\colon A\to B$ in $\cF(B)$ is said to be a \emph{covering} if there exists a monadic extension $p$ such that $p^*(A,f)$ is a trivial covering. One can also express this property by saying that ``$(A,f)$ is split by $p$''.

In the present work we are interested in those Galois structures where, moreover, $H,I$ present \cX\ as a full reflective subcategory of \cC \, or, equivalently, where $\epsilon\colon IH \to 1_\cX$ is an isomorphism. An important consequence of this assumption is that admissibility amounts then to asking that the functor $I$ preserves pullbacks of the form
\[
\xymatrix{
B \times_{HI(B)} H(X) \ar[r] \ar[d] & H(X) \ar[d]^{H(\phi)} \\
B \ar[r]_-{\eta_B} & HI(B)
}
\]
with $\phi$ in $\Phi$. This allows one to view the trivial coverings of our Galois structure as part of a relative factorization system for the class \cF, as explained below.
\medskip

In \cite{Ch04}, the author proposes a notion of factorization system relative to a given subclass of arrows. Namely, given a category \cC\ and a class  \cF\ of arrows in \cC\ containing identities, closed under composition, and pullback stable, a factorization system for \cF\ is a pair of classes of maps $(\cE,\cM)$ such that:
\begin{enumerate}
\item \cE\ and \cM\ both contain identities and are closed under composition with isomorphisms;
\item \cE\ and \cM\ are orthogonal to each other: for any commutative square in \cC
\[
\xymatrix{
\ar[r]^e \ar[d]_f & \ar[d]^g \ar@{-->}[dl]^d \\
\ar[r]_m &
}
\]
with $e\in\cE$, $m\in \cM$ there exists a unique $d$ such that $de=f$ and $md=g$;
\item \cM\ is contained in \cF;
\item every arrow $f$ in \cF\ is the composite $f=me$ of an $m$ in \cM\ with an $e$ in \cE. 
\end{enumerate}

Let $(\cC,\cX,I,H,\eta,\epsilon,\cF,\Phi)$ be a Galois structure in which $\epsilon\colon IH \to 1_\cX$ is an isomorphism and \cC\ is finitely complete. We shall denote by \cE\ the class of arrows in \cC\ which are inverted by $I$, and by \cM\ the class of trivial coverings (which are now those arrows $m$ in \cF\ such that diagram (\ref{diag:triv.cov}) is a pullback).

It is explained in \cite{Ch04} that, as it happens for all maps in the absolute case, the pair $(\cE,\cM)$ forms a factorization system for the subclass \cF. Then, again, one can consider the class $\cM^*$ of morphisms which are locally in \cM, i.e.\ the coverings of our Galois structure, and ask whether $(\cE',\cM^*)$ is a factorization system for \cF, where $\cE'$, as above, is the class of morphisms which are stably in \cE. A morphism is said to be stably in \cE\ if any pullback of it along any arrow in \cF\ is in \cE. In this case one says that this is a \emph{relative monotone-light factorization system for \cF.}

\section{Internal groupoids and the connected component functor} \label{sec:conn.comp}

Given an internal groupoid \bX\ in an exact category \cC, we write
\[
\xymatrix{
X_1 \ar@<1ex>[r]^{c} \ar@<-1ex>[r]_{d} & X_0 \ar[l]|e 
}
\]
for its underlying reflexive graph. Then the image $\pi_0(\bX)$ of $\bX$ through the reflection (\ref{adj:Gpd-C}) is defined as the codomain of the coequalizer of $(d,c)$.

We can look at the semi-left-exact reflection (\ref{adj:Gpd-C}) as the composite of two adjunctions:
\begin{equation} \label{diag:double.adj}
\begin{aligned}
\xymatrix{
\Gpd(\cC) \ar@<1.2ex>[r]^-{\Supp}_-\bot & \sfEq(\cC) \ar@<1.2ex>[r]^-{Q}_-\bot \ar@<1.2ex>[l]^-U & \cC \ar@<1.2ex>[l]^-{D'}
}
\end{aligned}
\end{equation}
where $\sfEq(\cC)$ is the full subcategory of $\Gpd(\cC)$ whose objects are internal equivalence relations, $U$ the inclusion functor (that we will drop where confusion in the notation is unlikely) and $\Supp$ its left adjoint sending every internal groupoid to the kernel pair of the coequalizer of $(d,c)$. $D'$ denotes the unique factorization of $D$ through $\sfEq(\cC)$. A functor $F=(f_0,f_1)\colon\bX\to\bY$ between internal groupoids can be conveniently displayed as
\begin{equation} \label{diag:functor}
\begin{aligned}
\xymatrix{
X_1 \ar[r]^{f_1} \ar@<1ex>[d]^c \ar@<-1ex>[d]_d & Y_1 \ar@<1ex>[d]^c \ar@<-1ex>[d]_d \\
X_0 \ar[u]|e \ar[r]_{f_0} & Y_0 \ar[u]|e
}
\end{aligned}
\end{equation}
For each internal groupoid \bX, we will denote the unit of the adjunction $\Supp\dashv U$ as follows:
\[
\xymatrix{
X_1 \ar[r]^-{\sigma_X} \ar@<1ex>[d]^c \ar@<-1ex>[d]_d & \Sigma X_1 \ar@<1ex>[d]^{r_2} \ar@<-1ex>[d]_{r_1} \\
X_0 \ar[u]|e \ar[r]_{1_{X_0}} & X_0 \ar[u]|{s_0}
}
\]
A functor $F$ is said to be a(n internal) \emph{discrete fibration} if the square $cf_1=f_0c$ in (\ref{diag:functor}) is a pullback.

Internal functors in the class \cM\ are also called ``$\pi_0$-cartesian'' by Bourn and the next result, which will be useful later on, is a reformulation of his observation on page 217 in \cite{B87}.

\begin{Lemma} \label{lem:M}
A functor $F$ between internal groupoids is in the class \cM\ if and only if both $F$ and $\Supp F$ are discrete fibrations. 
\end{Lemma}

We now recall the definition and some useful properties of the \emph{shift} or \emph{d\'ecalage} functor, first introduced in \cite{Ill}. It is defined on simplicial objects, in general, although here we are interested in its restriction to groupoids, as presented in \cite{B87}.

Given an internal groupoid \bX\ in \cC, the nerve functor sends it to a simplicial object in \cC
\[
\ldots\quad
\xymatrix{
X_3 \ar@<2.1ex>[r]^{d_0} \ar@<.7ex>[r]|{d_1} \ar@<-.7ex>[r]|{d_2} \ar@<-2.1ex>[r]_{d_3} & X_2 \ar@<1.4ex>[r]^{p_1} \ar[r]|{m} \ar@<-1.4ex>[r]_{p_2} & X_1 \ar@<1ex>[r]^c \ar@<-1ex>[r]_d & X_0 \ar[l]|e
}
\]
where $X_n$ for $n\geq 2$ is suitably defined as the object of composable sequences of arrows of length $n$ and face and degeneracy maps can be constructed by means of the structural morphisms of the groupoid. So, for example $X_2=X_1\times_{(d,c)}X_1$ and $m$ is the composition arrow, while $X_3=X_2\times_{(p_2,p_1)}X_2$ and the equation $md_1=md_2$ represents the associativity property of $m$. With a little abuse of notation, we will denote by \bX\ either the groupoid or its nerve.

The \emph{d\'ecalage} $\Dec(\bX)$ of \bX\ is the simplicial object obtained by shifting the $X_i$'s and dropping the last face and degeneracy maps at each level. This induces in a canonical way a morphism of simplicial objects: $\epsilon(\bX)\colon\Dec(\bX)\to\bX$
\[
\xymatrix@!=6ex{
\Dec(\bX) \ar[d]_{\epsilon(\bX)} & \qquad\qquad\ldots & X_4 \ar@<2.1ex>[r]^-{d_0} \ar@<.7ex>[r]|-{d_1} \ar@<-.7ex>[r]|-{d_2} \ar@<-2.1ex>[r]_-{d_3} \ar[d]_{d_4} & X_3 \ar@<1.4ex>[r]^{d_0} \ar[r]|{d_1} \ar@<-1.4ex>[r]_{d_2} \ar[d]_{d_3} & X_2 \ar@<1ex>[r]^{p_1} \ar@<-1ex>[r]_{m} \ar[d]_{p_2} & X_1 \ar[l]|{(1,ed)} \ar[d]^{d} \\
\bX & \qquad\qquad\ldots & X_3 \ar@<2.1ex>[r]^-{d_0} \ar@<.7ex>[r]|-{d_1} \ar@<-.7ex>[r]|-{d_2} \ar@<-2.1ex>[r]_-{d_3} & X_2 \ar@<1.4ex>[r]^{p_1} \ar[r]|{m} \ar@<-1.4ex>[r]_{p_2} & X_1 \ar@<1ex>[r]^c \ar@<-1ex>[r]_d & X_0 \ar[l]|e
}
\]
In fact, $\Dec(-)$ defines in the obvious way an endofunctor on the category of simplicial objects in \cC, which restricts to $\Gpd(\cC)$, and $(d,p_2)$ are the components of an internal functor that we also denote by $\epsilon(\bX)\colon\Dec(\bX)\to\bX$. Notice that, for an internal groupoid \bX, $\Dec(\bX)$ is an internal equivalence relation, since $X_2=X_1\times_{(d,c)}X_1$ and its projections $p_1$ and $m$ on $X_1$ form a kernel pair of $c$.

In what follows, it is also convenient to regard $\Gpd(\cC)$ as a full subcategory of the category $\Simpl_{3,i}(\cC)$ of 3-truncated simplicial objects in \cC, endowed with an ``inversion of arrows'' morphism $i\colon X_1\to X_1$, satisfying the usual equations. The embedding is obtained by sending each groupoid to the truncation of its nerve at the object of composable triples of arrows. $\Simpl_{3,i}(\cC)$ being a functor category, effective descent morphisms in this category are the same as level-wise effective descent morphisms in \cC. This observation has the following fruitful consequence. For the reader's convenience we provide a proof of it, although it is actually a special case of Proposition 3.2.4 in \cite{LeC}. 

\begin{Proposition} \label{prop:dec.desc}
The morphism $\epsilon(\bX)\colon\Dec(\bX)\to\bX$ is an effective descent morphism in $\Gpd(\cC)$.
\end{Proposition}

\begin{proof}
First of all observe that all the components of $\epsilon(\bX)$, regarded as a morphism in $\Simpl_{3,i}(\cC)$, are split epimorphisms, hence effective descent morphisms in \cC. In other words, $\epsilon(\bX)$ is effective for descent in $\Simpl_{3,i}(\cC)$. Now, by Corollary 3.9 in \cite{JST}, the thesis follows provided we show that for every pullback in $\Simpl_{3,i}(\cC)$ of the form
\[
\xymatrix{
\bW \ar[r] \ar[d] & \bZ \ar[d] \\
\Dec(\bX) \ar[r]_-{\epsilon(\bX)} & \bX
}
\]
if \bW\ is a groupoid, then \bZ\ is also a groupoid.

First, we have to prove that $Z_2\cong Z_1 \times_{(d_1,d_0)} Z_1$. Consider the following commutative cube:
\[
\xymatrix@=4ex{
W_2 \ar[rr] \ar[dd]_{d_0} \ar[dr]_{d_2} & & Z_2 \ar[dd]_(.3){d_0} \ar[dr]^{d_2} \\
& W_1 \ar[rr] \ar[dd]_(.3){d_0} & & Z_1 \ar[dd]^{d_0} \\
W_1 \ar[dr]_{d_1} \ar[rr] & & Z_1 \ar[dr]_(.35){d_1} \\
& W_0 \ar[rr] & & Z_0
}
\]
By construction, we may interpret the left hand side face as the image of the right hand side face through the change-of-base functor $d^*\colon\cC\downarrow X_0 \to \cC\downarrow X_1$. But since $d$ is a split epimorphism, $d^*$ is conservative, hence it reflects finite limits and the right hand side square is a pullback because so is the left hand side square. A similar argument can be used to prove that $Z_3\cong Z_2\times_{(d_2,d_0)}Z_2$ and this suffices to show that \bZ\ is an internal groupoid in \cC.
\end{proof}

We are now ready to describe the class $\cM^*$ of covering morphisms with respect to the absolute Galois structure induced by the adjunction (\ref{adj:Gpd-C}). It turns out that, also in this general context, they are characterized as in Theorem 3.2 of \cite{G01}. It is worth recalling that the covering morphisms for the case $\cC=\Set$ are well known to coincide with the classically defined covering maps between ordinary groupoids (described in \cite{GZ,Higg,Brown}, for example).

\begin{Proposition} \label{prop:df}
A morphism in $\Gpd(\cC)$ is in the class $\cM^*$ if and only if it is an internal discrete fibration.
\end{Proposition}

\begin{proof}
The ``only if'' part can be proved as in \cite{G01}, since the additional assumptions there play no role in this part of the proof. For the ``if'' part, one can use the property 2 at page 210 in \cite{B87}. Namely, if an internal functor $F\colon\bX\to\bY$ is a discrete fibration, then $\Dec(F)$ being the pullback of $F$ along $\epsilon(\bY)$ by that property 2, $\Dec(F)$ is a discrete fibration. Therefore it is a trivial covering by Lemma \ref{lem:M}, since $\Supp(\Dec(F))=\Dec(F)$. This completes our proof since $\epsilon(\bY)$ is an effective descent morphism by Proposition \ref{prop:dec.desc}.
\end{proof}

In fact, the last proof also shows that the morphism $\epsilon(\bX)\colon\Dec(\bX)\to\bX$ is such that if a morphism of codomain \bX\ is a covering, then it is split by $\epsilon(\bX)$. Let us rephrase this fact as:

\begin{Proposition}
The morphism $\epsilon(\bX)\colon\Dec(\bX)\to\bX$ is a universal cover for \bX.
\end{Proposition}

It is shown in \cite{B87} that, in every Barr-exact category \cC, internal discrete fibrations are part of a factorization system on $\Gpd(\cC)$, where the internal functors in the corresponding orthogonal class are called \emph{final}. This is an internal version of the \emph{comprehensive} factorization system introduced in \cite{SW}. It is then clear that the pair $(\cE',\cM^*)$ is a factorization system if and only if the class $\cE'$ of morphisms stably in \cE\ coincide with the class of final functors. Example \ref{final.not.pbs} below shows that, in fact, final functors are not stably in \cE, so $(\cE,\cM)$ does not admit an associated monotone-light factorization system. Let us notice that, on the contrary, the left adjunction in (\ref{diag:double.adj}) is a restriction of the one considered in \cite{Xarez}, hence it admits a monotone-light factorization system.

The following proposition is the translation in terms of groupoids of Corollary 5.2 in \cite{CMM14} (thanks to the equivalence between internal groupoids and internal crossed modules \cite{J03}) and provides a characterization of final functors in a semi-abelian context. We recall that in this context the functor $\pi_1$ associates, with every groupoid $\bX$, the kernel of the morphism $(d,c)\colon X_1\to X_0\times X_0$.

\begin{Proposition} \label{char.final.semiab}
A functor $F\colon\bX\to\bY$ between internal groupoids in a semi-abelian category \cC\ is final if and only if $\pi_0(F)$ is an isomorphism and $\pi_1(F)$ is a regular epimorphism.
\end{Proposition}

\begin{Example} \label{final.not.pbs}
Let $A$ be a non-trivial abelian group. Consider the following pullback in $\Gpd(\Ab)$:
\[
\xymatrix{
& 0 \ar[rr] \ar@<1ex>[dd] \ar@<-1ex>[dd] \ar[dl] & & A \ar@<1ex>[dd]^{1_A} \ar@<-1ex>[dd]_{1_A} \ar[dl]_{(0,(1_A,1_A))} \\
A \ar[rr]^(.55){(1_A,0)} \ar@<1ex>[dd]^{1_A} \ar@<-1ex>[dd]_{1_A} & & A \times (A \times A) \ar@<1ex>[dd]^(.3){1_A \times p_2} \ar@<-1ex>[dd]_(.3){1_A \times p_1} \\
& 0 \ar[rr] \ar[dl] \ar[uu] & & A \ar[dl]^{(0,1_A)} \ar[uu] \\
A \ar[rr]_-{(1_A,0)} \ar[uu] & & A \times A \ar[uu]
}
\]
where $p_1$ and $p_2$ are the projections of the product $A \times A$. Applying $\pi_1$ we obtain the trivial square on the left here below (all the groupoids involved being relations), while applying $\pi_0$ we have the right hand side square:
\[
\xymatrix{
0 \ar[r] \ar[d] & 0 \ar[d] & 0 \ar[r] \ar[d] & A \ar[d]^0 \\
0 \ar[r] & 0 & A \ar[r]_{1_A} & A
} 
\]
By Proposition \ref{char.final.semiab}, the front face of the cube above is a final functor, but a pullback of it (the back face) is not inverted by $\pi_0$. In other words, final functors are not stably in the class \cE\ of functors inverted by $\pi_0$.
\end{Example}

\section{The comprehensive factorization as a relative monotone-light factorization system} \label{sec:ex.mal}

It is explained in \cite{G99} that internal groupoids or, equivalently, internal categories, in an exact Mal'tsev category \cC\ have some particularly nice exactness properties, since they form again an exact Mal'tsev category. Moreover, in the subsequent paper \cite{G01}, the reflection $\pi_0\colon\Gpd(\cC)\to\cC$ is considered as part of a Galois structure where the chosen classes of morphisms are the regular epimorphisms (also called ``extensions''). In fact, in this context, \cC\ turns out to be a Birkhoff subcategory of $\Gpd(\cC)$, and the general theory of central extensions developed in \cite{JK} applies to this case. Trivial coverings (or trivial extensions) and coverings (or central extensions) with respect to the above Galois structure are characterized in \cite{G01}.

\begin{Proposition}[\cite{G01}] \label{prop:M*}
A regular epimorphic functor $F$ in the category $\Gpd(\cC)$ of internal groupoids in an exact Mal'tsev category $\cC$ is
\begin{itemize}
\item a \emph{trivial covering} if and only if both $F$ and $\Supp F$ are discrete fibrations;
\item a \emph{covering} if and only if $F$ is a discrete fibration.
\end{itemize}
\end{Proposition}

The class \cM\ of trivial coverings, together with the class \cE\ of functors inverted by $\pi_0$, form a factorization system in the sense of \cite{Ch04} for the class \cF\ of regular epimorphic internal functors in \cC. What is nice here is that this factorization system can be simultaneously stabilised and localised to yield a corresponding monotone-light factorization system relative to \cF.
\medskip

Given a functor $F\colon\bX\to\bY$ between internal groupoids, consider the comparison arrow $\phi_F$ in the diagram
\[
\xymatrix@=2ex{
X_1 \ar[ddd]_{(d,c)} \ar[rrr]^{f_1} \ar[dr]_(.35){\phi_F} & & & Y_1 \ar[ddd]^{(d,c)} \\
& (X_0\times X_0) \times_{(Y_0\times Y_0)} Y_1 \ar[ddl] \ar[urr] \\
\\
X_0 \times X_0 \ar[rrr]_{f_0 \times f_0} & & & Y_0 \times Y_0
}
\]
We recall that $F$ is said to be
\begin{itemize}
\item \emph{full} when $\phi_F$ is a regular epimorphism;
\item \emph{faithful} when $\phi_F$ is a monomorphism;
\item \emph{fully faithful} when $\phi_F$ is an isomorphism;
\item \emph{essentially surjective on objects} when $\pi_0(F)$ is a regular epimorphism.
\end{itemize}
It is easy to prove that the class of full functors is pullback stable and that if a functor $F$ is full, then $\Supp F$ is fully faithful. Full functors also enjoy the following useful property:

\begin{Lemma} \label{lemma:supp.full}
The functor $\Supp\colon\Gpd(\cC)\to\sfEq(\cC)$ preserves the pullbacks of full functors along any functor.
\end{Lemma}

\begin{proof}
Let the diagram below be a pullback in $\Gpd(\cC)$.
\[
\xymatrix{
\bW \ar[r]^-{\overline{F}} \ar[d]_{\overline{G}} & \bZ \ar[d]^G \\
\bX \ar[r]_F & \bY
}
\]
If $F$ is full, so is $\overline{F}$, hence both $\Supp(F)$ and $\Supp(\overline{F})$ are fully faithful. It is then clear that in the cube
\[
\xymatrix@=4ex{
\Sigma W_1 \ar[rr]^{\Sigma\overline{f}_1} \ar[dd]_{(r_1,r_2)} \ar[dr]_{\Sigma\overline{g}_1} & & \Sigma Z_1 \ar[dd]_(.3){(r_1,r_2)} \ar[dr]^{\Sigma g_1} \\
& \Sigma X_1 \ar[rr]_(.35){\Sigma f_1} \ar[dd]_(.3){(r_1,r_2)} & & \Sigma Y_1 \ar[dd]^{(r_1,r_2)} \\
W_0 \times W_0 \ar[dr]_{\overline{g}_0\times \overline{g}_0} \ar[rr]_(.7){\overline{f}_0\times \overline{f}_0} & & Z_0 \times Z_0 \ar[dr]^{g_0\times g_0} \\
& X_0 \times X_0 \ar[rr]_{f_0 \times f_0} & & Y_0 \times Y_0
}
\]
the top face is a pullback by cancellation since the bottom, back and front face are.
\end{proof}

The next two results are specific of Mal'tsev or Goursat categories and will be useful later on. Recall that, in a non-exact context, the functor $Q$ in (\ref{diag:double.adj}) is only defined on the subcategory $\EffEq(\cC)$ of effective equivalence relations in \cC.

\begin{Proposition} \label{prop:goursat}
A regular category \cC\ is Goursat if and only if the functor $Q\colon\EffEq(\cC)\to\cC$ preserves exact forks.
\end{Proposition}

\begin{proof}
This follows easily from the fact that, for regular categories, the Goursat property is equivalent to the lower $3\times 3$ denormalised lemma, as proved in \cite{GR12, Lack}.
\end{proof}

\begin{Lemma} \label{lemma:pi0.pb}
Let \cC\ be an exact Mal'tsev category, and let the left hand side downward diagram below be a pullback in $\Gpd(\cC)$, where the vertical arrows are split epimorphisms. Then the comparison arrow $u$ in the induced diagram on the right hand side is a regular epimorphism.
\[
\xymatrix{
\bW \ar[r] \ar@<.5ex>[d] & \bZ \ar@<.5ex>[d] \\
\bX \ar[r] \ar@<.5ex>[u] & \bY \ar@<.5ex>[u]
}
\qquad\qquad
\xymatrix@=2ex{
\pi_0(\bW) \ar[ddd] \ar[rrr] \ar[dr]_(.35){u} & & & \pi_0(\bZ) \ar[ddd] \\
& \pi_0(\bX) \times_{\pi_0(\bY)} \pi_0(\bZ) \ar[ddl] \ar[urr] \\
\\
\pi_0(\bX) \ar[rrr] & & & \pi_0(\bY)
}
\]
\end{Lemma}

\begin{proof}
By Lemma 4.1 in \cite{B03}, the canonical morphism $W_0 \to \pi_0(\bX) \times_{\pi_0(\bY)} \pi_0(\bZ)$ is a regular epimorphism, and it factors through $u$, which is then a regular epimorphism.
\end{proof}

\begin{Theorem} \label{thm:cf.ml}
Let \cC\ be an exact Mal'tsev category. The adjunction
\[
\xymatrix{
\Gpd(\cC) \ar@<1.2ex>[r]^-{\pi_0}_-\bot & \cC \ar@<1.2ex>[l]^-D
}
\]
yields a monotone-light factorization system for regular epimorphisms in the category $\Gpd(\cC)$ of internal groupoids in \cC.
\end{Theorem}

\begin{proof}
We are going to prove this result by means of a relative version of Proposition 6.10 in \cite{CJKP} (see Observation 2.2 in \cite{Ch04}). In order to do this, we have to show that $\Gpd(\cC)$ has enough (relatively) stabilising objects, i.e.\ that for each \bY\ in $\Gpd(\cC)$ there exists an effective descent morphism $P\colon\bE\to\bY$ such that the $(\cE,\cM)$ factorization of any morphism with codomain \bE\ is (relatively) stable. In fact, for each \bY\ in $\Gpd(\cC)$, the needed stabilising object will be chosen to be $\Dec(\bY)$, together with the morphism $\epsilon(\bY)\colon\Dec(\bY)\to\bY$.

Let us consider a morphism $F\colon\bX\to\Dec(\bY)$ in $\Gpd(\cC)$, and let
\[
\xymatrix{
\bX \ar[r]^-{E} & \bW \ar[r]^-{M} & \Dec(\bY)
}
\]
be its $(\cE,\cM)$ factorization.
We have to show that any pullback of $E$ along a regular epimorphism is in the class \cE. Recall from Section \ref{sec:conn.comp} that $\Dec(\bY)$ is an internal equivalence relation, hence \bW\ is also a relation because $M$ is a discrete fibration. Using the fact that \bW\ is an equivalence relation and $\pi_0(E)$ is an
isomorphism, it is easy to show that $E$ is full.

Let us now consider the following pullback in $\Gpd(\cC)$, where $P$ is a regular epimorphism:
\[
\xymatrix{
\bX\times_{\bW}\bZ \ar[r]^-{\overline{E}} \ar[d]_{\overline{P}} & \bZ \ar[d]^P \\
\bX \ar[r]_E & \bW
}
\]
By Lemma \ref{lemma:supp.full}, this pullback is preserved by $\Supp$. Then we can take the kernel pairs $\bR$ and $\bS$ of the corresponding vertical arrows in $\Eq(\cC)$ and get the discrete fibration in $\Eq(\cC)$, which is the upper part of the first of the diagrams
\[
\xymatrix@=7ex{
\bR \ar@<1ex>[d] \ar@<-1ex>[d] \ar[r] & \bS \ar@<1ex>[d] \ar@<-1ex>[d] \\
\Supp(\bX\times_{\bW}\bZ) \ar[r]^-{\Supp(\overline{E})} \ar[d]_{\Supp(\overline{P})} \ar[u] & \Supp(\bZ) \ar[d]^{\Supp(P)} \ar[u] \\
\Supp(\bX) \ar[r]_{\Supp(E)} & \Supp(\bW)
}
\qquad
\xymatrix@=7ex{
Q(\bR) \ar@<1ex>[d] \ar@<-1ex>[d] \ar[r]^v \ar@{}[dr]|{(\ast\ast)} & Q(\bS) \ar@<1ex>[d] \ar@<-1ex>[d] \\
\pi_0(\bX\times_\bW\bZ) \ar[r]^-{\pi_0(\overline{E})} \ar[d]_{\pi_0(\overline{P})} \ar@{}[dr]|{(\ast)} \ar[u] & \pi_0(\bZ) \ar[d]^{\pi_0(P)} \ar[u] \\
\pi_0(\bX) \ar[r]_{\pi_0(E)} & \pi_0(\bW)
}
\]
We apply now the functor $Q \colon \Eq(\cC) \rightarrow $ $\cC$ to that diagram on the left, to get the diagram on the right, where, by Proposition \ref{prop:goursat}, the two columns are exact forks. Moreover, $\pi_0(\overline{E})$ is monomorphic since $\Supp(\overline{E})$ is fully faithful (see Proposition 1.1 in \cite{B03}, for instance). Hence $v$ is also a monomorphism and by Lemma \ref{lemma:pi0.pb} the two downwarded commutative squares $(\ast\ast)$ are pullbacks. By elementary descent theory, $(\ast)$ is then a pullback. Then, since $\pi_0(E)$ is an isomorphism, so is also $\pi_0(\overline{E})$, thus proving that $\overline{E}$ is in \cE.
\end{proof}

\begin{Remark}
{\normalfont
Looking at the proof of Theorem \ref{thm:cf.ml}  more closely one can observe that, in fact, the key point is that the internal functor $E$ is full and $\pi_0(E)$ is an isomorphism, or equivalently $E$ is full and essentially surjective on objects, and both properties are preserved by taking a pullback along a regular epimorphism. By the characterisation of final functors given in \cite[Theorem 4.2]{Ci17}, this means that final functors are in the class $\cE'$ of morphisms stably in $\cE$ by taking pullbacks along regular epimorphisms. Conversely, by Lemma 4.1 in \cite{Ci17}, any morphism in $\cE'$ is in particular a final functor. Combining this result with Proposition \ref{prop:M*} we deduce that the relative monotone-light factorization system of Theorem \ref{thm:cf.ml} is exactly the comprehensive factorization system restricted to regular epimorphisms.
}
\end{Remark}



\begin{thebibliography}{99}

\bibitem{B87} D. Bourn, The shift functor and the comprehensive factorization for internal groupoids, Cah. Topol. G{\'e}om. Diff{\'er}. Cat{\'e}g. 28 (1987) 197--226. 

\bibitem{B03} D. Bourn, The denormalized $3 \times 3$ lemma, J. Pure Appl. Algebra 177 (2003) 117--129.

\bibitem{Brown} R. Brown, Topology and groupoids, Third edition of Elements of modern topology [McGraw-Hill, New York, 1968], BookSurge, LLC, Charleston, SC, 2006.

\bibitem{CJKP} A. Carboni, G. Janelidze, G. M. Kelly, and R. Par\'e, On localization and stabilization of factorization systems, Appl. Categ. Struct. 5 (1997) 1--58.

\bibitem{CHK} C. Cassidy, M. H\'ebert, and G. M. Kelly, Reflective subcategories, localizations and factorizations systems, J. Austral. Math. Soc. 38 (1985) 287--329.

\bibitem{Ch04} D. Chikhladze, Monotone-light factorization for Kan fibrations of simplicial sets with respect to groupoids, Homol. Homot. Appl. 6 (2004) 501--505.

\bibitem{Ci17} A. S. Cigoli, A characterization of final functors between internal groupoids in exact categories, preprint arXiv:1711.10747.

\bibitem{CMM14} A. S. Cigoli, S. Mantovani, and G. Metere, A push forward construction and the comprehensive factorization for internal crossed modules, Appl. Categ. Structures 22 (2014) 931--960.

\bibitem{EG13} T. Everaert and M. Gran, Monotone-light factorisation systems and torsion theories, Bull. Sci. Math. 137 (2013) 996--1006.

\bibitem{GZ} P. Gabriel and M. Zisman, Calculus of fractions and homotopy theory, Springer, Berlin, 1967. 

\bibitem{G99} M. Gran, Internal categories in Mal'cev categories, J. Pure Appl. Algebra 143 (1999) 221--229.

\bibitem{G01} M. Gran, Central extensions and internal groupoids in Maltsev categories, J. Pure Appl. Algebra 155 (2001) 139--166.

\bibitem{GR12} M. Gran and D. Rodelo, A new characterisation of Goursat categories, Appl. Categ. Structures 20 (2012) 229--238.

\bibitem{Higg} P. J. Higgins, Categories and Groupoids. Reprint on the original Notes on categories and groupoids [Van Nostrand Reinhold, London, 1971]. Repr. Theory Appl. Categ. 7 (2005) 1--195.

\bibitem{Ill} L. Illusie, Complexe cotangent et d\'eformations II, Lecture Notes in Mathematics 283, Springer-Verlag, Berlin-New York, 1972. vii+304 pp.

\bibitem{J89} G. Janelidze (G. Z. Dzhanelidze), The fundamental theorem of Galois theory, Math. USSR-Sb. 64 (1989) 359--374.

\bibitem{J90} G. Janelidze, Pure Galois Theory in categories, J. Algebra 132 (1990) 270--286.

\bibitem{J03} G. Janelidze, Internal crossed modules, Georgian Math. J. 10 (2003) 99--114.

\bibitem{JK} G. Janelidze and G. M. Kelly, Galois theory and a general notion of central extension, J. Pure Appl. Algebra 97 (1994) 135--161.

\bibitem{JST} G. Janelidze, M. Sobral, and W. Tholen, Beyond Barr exactness: effective descent morphisms, in: M.C. Pedicchio, W. Tholen (Eds.), Categorical Foundations: Special Topics in Order, Topology, Algebra and Sheaf Theory, in: Encyclopedia of Math. Appl., vol. 97, Cambridge University Press, 2004, pp. 359--405.

\bibitem{Lack} S. Lack, The $3$-by-$3$ lemma for Goursat categories, Homology, Homotopy Appl. 6 (2004) 1--3.

\bibitem{LeC} I. Le Creurer, Descent of internal categories, Ph.D. thesis, Universit\'e catholique de Louvain, 1999.

\bibitem{SW} R. Street and R. F. C. Walters, The comprehensive factorization of a functor, Bull. Amer. Math. Soc. 79 (1973) 936--941.

\bibitem{Xarez} J. J. Xarez, Internal monotone-light factorization for categories via preorders, Theory Appl. Categ. 15 (2004) 235--251.

\end{thebibliography}
\end{document}